\documentclass[preprint,12pt,3p]{elsarticle}

\usepackage{amsmath,amssymb,amsthm,mathtools}
\usepackage[hidelinks]{hyperref}

\biboptions{sort&compress}

\newtheorem{theorem}{Theorem}[section]
\newtheorem{lemma}[theorem]{Lemma}
\newtheorem{proposition}[theorem]{Proposition}
\newtheorem{corollary}[theorem]{Corollary}
\newtheorem{conjecture}[theorem]{Conjecture}
\theoremstyle{definition}
\newtheorem{definition}[theorem]{Definition}
\newtheorem{example}[theorem]{Example}
\theoremstyle{remark}
\newtheorem{remark}[theorem]{Remark}
\newtheorem*{maintheorem}{\bf Main Theorem}
\newcommand{\LC}{\operatorname{LC}}
\newcommand{\LT}{\operatorname{LT}}
\newcommand{\LM}{\operatorname{LM}}
\newcommand{\rank}{\operatorname{rank}}

\begin{document}
	
	\begin{frontmatter}
		
		\title{Reduction of the Finite Generation Problem for Leading Term
			Ideals under Arbitrary Rational Monomial Orders to the Lexicographic}
		
		\author[stu]{Xiaopeng Zheng\corref{cor1}}
		\ead{xiaopengzheng@stu.edu.cn}
		\cortext[cor1]{Corresponding author}
		
		\address[stu]{College of Mathematics and Computer Science, Shantou University,
			Shantou 515821, China}
		
		\begin{abstract}
			A commutative ring $R$ is called a Gr\"obner ring if, for every
			$n\geq 1$, the leading term ideal of every finitely generated ideal
			of $R[X_1,\ldots,X_n]$ is finitely generated with respect to the
			lexicographic order $X_1\succ\cdots\succ X_n$. We prove that this
			property is equivalent to the
			condition that, for every $n\geq 1$, every rational monomial order
			$\prec$ on $R[X_1,\ldots,X_n]$, and every finitely generated ideal
			$I\subseteq R[X_1,\ldots,X_n]$, the leading term ideal
			$\operatorname{LT}_{\prec}(I)$ is finitely generated. The
			construction uses a tagged monomial embedding, a compatible group
			grading, and dehomogenization. As an application, we apply this reduction to valuation rings and
			determine when the finite generation property holds for every rational
			monomial order. In particular, for valuation domains, this gives a proof
			of the rational monomial order version of the Gr\"obner ring conjecture.
		\end{abstract}

	\end{frontmatter}
	

	\section{Introduction}
	
	Let $R$ be a commutative ring with identity, let
	$A=R[X_1,\ldots,X_n]$, and fix a monomial order $\prec$ on the
	monomials of $A$. For an ideal $I\subseteq A$, write
	$\LT_\prec(I)=\langle \LT_\prec(f)\mid 0\neq f\in I\rangle$ for its
	leading term ideal. When $I$ is finitely generated, finite generation
	of $\LT_\prec(I)$ is equivalent to the existence of a finite
	Gr\"obner basis of $I$ with respect to $\prec$. This finiteness is
	automatic over Noetherian coefficient rings, but it can fail over
	non-Noetherian rings even when $I$ itself is finitely generated
	\cite{AdamsLoustaunau1994,Pauer2007,MonceurYengui2012,
		PolaYengui2012Negative}.
	
	A ring $R$ is
	$n$-Gr\"obner if every finitely generated ideal of
	$R[X_1,\ldots,X_n]$ has a finitely generated leading term ideal for
	the lexicographic order $X_1\succ\cdots\succ X_n$, and $R$ is a
	Gr\"obner ring if it is $n$-Gr\"obner for every $n\geq 1$
	\cite[Definition~215]{Yengui2015Book}. The definition is formulated
	there in a constructive setting for strongly discrete coherent
	rings. The underlying finiteness property makes sense for every
	commutative ring, and we use the same terminology in that generality.
	Thus, throughout this paper, the expression ``Gr\"obner ring'' refers
	to the lexicographic definition from Yengui's book. This convention
	is slightly different from the earlier existential elimination-order
	formulation in \cite{PolaYengui2014}.
	
	The Leading Terms Ideals Conjecture relates finite generation for
	arbitrary monomial orders to the Gr\"obner and $1$-Gr\"obner
	properties \cite[Conjecture~225]{Yengui2015Book}. A subsequent
	counterexample shows that the condition involving all monomial
	orders is false in general \cite{Yengui2021Counterexample}. We
	therefore formulate the following updated version by restricting
	this condition to rational monomial orders.
	
	\begin{conjecture}[The Leading Terms Ideals Conjecture
		(updated version)]
		\label{conj:updated-leading-terms}
		Let $R$ be a strongly discrete coherent ring. Then the following
		conditions are equivalent.
		\begin{enumerate}
			\item[(1)] For every $n\geq 1$, every rational monomial order
			$\prec$ on $R[X_1,\ldots,X_n]$, and every finitely generated
			ideal $I\subseteq R[X_1,\ldots,X_n]$, the leading term ideal
			$\LT_\prec(I)$ is finitely generated.
			
			\item[(2)] The ring $R$ is a Gr\"obner ring.
			
			\item[(3)] The ring $R$ is $1$-Gr\"obner.
		\end{enumerate}
	\end{conjecture}
	
	The main result of this paper proves the equivalence of
	assertions~{\rm(1)} and~{\rm(2)}. In fact, our result holds in
	greater generality than required in
	Conjecture~\ref{conj:updated-leading-terms}, since no discreteness
	or coherence assumption on the coefficient ring is needed.
	
	\begin{maintheorem}
		\label{thm:rational-grobner-characterization}
		Let $R$ be a commutative ring with identity. Then the following
		conditions are equivalent.
		\begin{enumerate}
			\item[(1)] For every $n\geq 1$, every rational monomial order
			$\prec$ on $R[X_1,\ldots,X_n]$, and every finitely generated
			ideal $I\subseteq R[X_1,\ldots,X_n]$, the leading term ideal
			$\LT_\prec(I)$ is finitely generated.
			
			\item[(2)] The ring $R$ is a Gr\"obner ring.
		\end{enumerate}
	\end{maintheorem}
	
	The proof of the main theorem relies on the matrix representation of
	rational monomial orders. We therefore recall the relevant
	terminology and the normalization used in our construction.
	A monomial order $\prec$ is called rational if there is a matrix $Q$
	with rational entries such that
	\[
	X^\alpha\prec X^\beta
	\quad\Longleftrightarrow\quad
	Q\alpha<_{\mathrm{lex}}Q\beta.
	\]
	Matrix representations of monomial orders go back to Robbiano
	\cite{Robbiano1985,Robbiano1986}, and the corresponding theory for
	monomial preorders is developed in \cite{KemperTrungAnh2018}.
	Rational monomial orders include the standard lexicographic, graded
	lexicographic, and graded reverse lexicographic orders. Every rational
	monomial order admits a representation by a nonnegative integer
	matrix $M\in\mathbb N^{n\times n}$ of rank $n$
	\cite{KemperYengui2020,GuyotYengui2024}. This is the form used in our
	construction.
	
The proof proceeds by a reduction for a fixed ideal. Let $\prec$ be a
rational monomial order on $A=R[X_1,\ldots,X_n]$, represented by a
matrix $M$, and let $I=\langle f_1,\ldots,f_s\rangle$. Set
$S=R[U_1,\ldots,U_n,X_1,\ldots,X_n]$. The construction is motivated
by the substitution $X^\alpha\mapsto U^{M\alpha}$ used by Guyot and
Yengui to transfer leading terms of individual polynomials
\cite{GuyotYengui2024}.  The present
argument modifies this substitution by retaining the original monomial
and introducing a compatible grading, which allows the reduction to be
carried out at the level of ideals. We define the tagged embedding
$\Psi_M(X^\alpha)=U^{M\alpha}X^\alpha$. Let $H$ be the ideal generated
by $\Psi_M(f_1),\ldots,\Psi_M(f_s)$. The grading $\deg_M(U^pX^\alpha)=p-M\alpha$ makes $H$ a graded ideal.
Let $\chi\colon S\to A$ be the homomorphism that fixes the variables
$X_i$ and sends every $U_j$ to $1$. For every nonzero homogeneous
element $P\in S$, one has
\[
\chi\!\left(\LT_{\prec_{\mathrm{lex}}}(P)\right)
=
\LT_\prec\!\left(\chi(P)\right),
\]
where $\prec_{\mathrm{lex}}$ is the lexicographic order determined by
$U_1\succ\cdots\succ U_n\succ X_1\succ\cdots\succ X_n$. This exact
identity allows a finite generating set of
$\LT_{\prec_{\mathrm{lex}}}(H)$ to descend to a finite generating set
of $\LT_\prec(I)$.
	
A principal source of examples is the theory of valuation rings. The
Gr\"obner ring conjecture can be traced back to Yengui's work on
dynamical Gr\"obner bases \cite{Yengui2006}. For valuation domains,
the one-variable Gr\"obner property is equivalent to having Krull
dimension at most one \cite{LombardiSchusterYengui2012}. Yengui later
proved the corresponding multivariable result for the lexicographic
order \cite{Yengui2014Lex}. More generally, Pola and Yengui proved
that a valuation ring, possibly with zero divisors, is a Gr\"obner
ring if and only if it is coherent and archimedean
\cite{PolaYengui2014}. Combined with our main theorem, these results
characterize the valuation rings for which every finitely generated
polynomial ideal has a finitely generated leading term ideal with
respect to every rational monomial order. In the domain case, this
yields the updated version of the Gr\"obner ring conjecture formulated
in \cite{Yengui2021Counterexample}.
	
	However, Lombardi, Neuwirth, and Yengui cite an unpublished 2024
	preprint by Yengui entitled \emph{A solution to the Gr\"obner ring
		conjecture} and describe it as a recent solution
	\cite{LombardiNeuwirthYengui2026}. Since this preprint was not
	available to us, we have not been able to compare the two approaches.
	Accordingly, the present paper does not claim to provide the first
	solution of the updated Gr\"obner ring conjecture. Its main
	contribution is the general reduction theorem over arbitrary
	commutative coefficient rings.
	
	The finite generation property also has a direct computational
	meaning. Over strongly discrete coherent rings, the generalized
	Buchberger algorithm converges precisely when the module of leading
	terms is finitely generated
	\cite{LombardiNeuwirthYengui2026}. Our result is a finiteness
	reduction rather than a complexity statement. It shows that, once the
	lexicographic Gr\"obner ring property is known in every finite number
	of variables, termination and finite Gr\"obner bases follow for all
	rational monomial orders.
	
	The remainder of the paper is organized as follows. Section~2
	recalls the required notions concerning leading term ideals,
	Gr\"obner rings, the Leading Terms Ideals Conjecture, matrix
	representations of rational monomial orders, and group gradings.
	Section~3 reduces the finite generation problem for leading term
	ideals under rational monomial orders to the lexicographic case in an
	enlarged polynomial ring. As a consequence, it derives the rational
	monomial order characterization of Gr\"obner rings.
	Section~4 derives the consequences for valuation domains and
	valuation rings. Section~5 concludes the paper.
	
	\section{Preliminaries}\label{sec:preliminaries}
	
	Throughout the paper, $R$ denotes a commutative ring with identity,
	$\mathbb{N}=\{0,1,2,\ldots\}$, and
	$A=R[X_1,\ldots,X_n]$. For
	$\alpha=(\alpha_1,\ldots,\alpha_n)\in\mathbb{N}^n$, we write
	$X^\alpha=X_1^{\alpha_1}\cdots X_n^{\alpha_n}$. The definitions and
	conventions in this section follow the standard treatments of
	Gr\"obner bases over rings and of matrix descriptions of monomial
	orders
	\cite{AdamsLoustaunau1994,Pauer2007,Robbiano1985,Robbiano1986,
		KemperTrungAnh2018,LombardiNeuwirthYengui2026,Yengui2015Book}.
	
	\subsection{Monomial orders and leading term ideals}
	
	\begin{definition}[Monomial order]\label{def:monomial-order}
		A monomial order $\prec$ on $A$ is a total order on the monomials
		$X^\alpha$, with $\alpha\in\mathbb{N}^n$, such that
		\begin{enumerate}
			\item $1\preceq X^\alpha$ for every $\alpha\in\mathbb{N}^n$,
			\item $X^\alpha\prec X^\beta$ implies
			$X^{\alpha+\gamma}\prec X^{\beta+\gamma}$ for every
			$\gamma\in\mathbb{N}^n$,
			\item every nonempty set of monomials has a least element.
		\end{enumerate}
	\end{definition}
	
	If the variables are ordered as
	$Y_1\succ\cdots\succ Y_r$, the corresponding lexicographic order
	compares exponent vectors from left to right. Thus
	$Y^a\prec_{\mathrm{lex}}Y^b$ when, at the first index $i$ for which
	$a_i\neq b_i$, one has $a_i<b_i$.
	
	Throughout the paper, $\prec_{\mathrm{lex}}$ denotes the
	lexicographic order on $R[X_1,\ldots,X_n]$ with variable priority
	$X_1\succ X_2\succ\cdots\succ X_n$. Let $0\neq f=\sum_{\alpha}c_\alpha X^\alpha\in A$, where only
	finitely many coefficients $c_\alpha$ are nonzero. If $X^{\alpha_0}$
	is the largest monomial in the support of $f$, we set
	\[
	\LM_\prec(f)=X^{\alpha_0},\qquad
	\LC_\prec(f)=c_{\alpha_0},\qquad
	\LT_\prec(f)=c_{\alpha_0}X^{\alpha_0}.
	\]
	The coefficient is part of the leading term. This distinction is
	essential over a general coefficient ring.
	
	\begin{definition}[Leading term ideal and Gr\"obner basis]
		\label{def:leading-term-ideal}
		For an ideal $I\subseteq A$, its leading term ideal with respect to
		$\prec$ is
		\[
		\LT_\prec(I)
		=
		\bigl\langle \LT_\prec(f)\mid 0\neq f\in I\bigr\rangle.
		\]
		A finite set $G\subseteq I\setminus\{0\}$ is called a Gr\"obner
		basis of $I$ with respect to $\prec$ if
		\[
		\LT_\prec(I)
		=
		\bigl\langle \LT_\prec(g)\mid g\in G\bigr\rangle.
		\]
	\end{definition}
	
	\begin{proposition}[Finite Gr\"obner bases and leading term ideals]
		\label{prop:finite-grobner-basis}
		Let $I\subseteq A$ be a finitely generated ideal. Then
		$\LT_\prec(I)$ is finitely generated if and only if $I$ admits a
		finite Gr\"obner basis with respect to $\prec$.
	\end{proposition}
	
	This elementary equivalence is standard in Gr\"obner basis theory
	over rings. See
	\cite{AdamsLoustaunau1994,Pauer2007,LombardiNeuwirthYengui2026}.
	
	\subsection{Gr\"obner rings and the Leading Terms Ideals Conjecture}
	
	We first recall the ring theoretic terminology used in the main
	statement. The following definition is the lexicographic convention
	adopted in \cite[Definition~215]{Yengui2015Book}. Although it is
	introduced there in the setting of strongly discrete coherent rings,
	the underlying finiteness property is meaningful for every
	commutative ring.
	
	\begin{definition}[Gr\"obner rings]
		\label{def:grobner-ring-book}
		Let $R$ be a commutative ring and let $n\geq 1$. The ring $R$ is
		called $n$-Gr\"obner if, for every finitely generated ideal
		$I\subseteq R[X_1,\ldots,X_n]$, the leading term ideal
		$\LT_{\prec_{\mathrm{lex}}}(I)$ is finitely generated with respect
		to the lexicographic order
		$X_1\succ\cdots\succ X_n$. The ring $R$ is called a Gr\"obner ring
		if it is $n$-Gr\"obner for every $n\geq 1$.
	\end{definition}
	
	Reversing the displayed variable ordering gives an equivalent
	property after renaming the variables. We use
	$X_1\succ\cdots\succ X_n$ throughout the paper. By
	Proposition~\ref{prop:finite-grobner-basis}, the $n$-Gr\"obner
	property is equivalent to the existence of a finite Gr\"obner basis
	with respect to this lexicographic order for every finitely generated
	ideal of $R[X_1,\ldots,X_n]$.
	
	\begin{proposition}
		\label{prop:noetherian-grobner}
		Every Noetherian ring is a Gr\"obner ring.
	\end{proposition}
	
	This follows from the Hilbert basis theorem and is included among the
	basic examples in \cite[Example~216]{Yengui2015Book}. The class of
	Gr\"obner rings is strictly larger than the class of Noetherian
	rings, since non-Noetherian Gr\"obner rings also exist.
	
	Recall that a ring is coherent if every finitely generated ideal is
	finitely presented. It is stably coherent if
	$R[X_1,\ldots,X_m]$ is coherent for every $m\geq 1$. The following
	result shows that the Gr\"obner property imposes a strong coherence
	condition on all finite polynomial extensions.
	
	\begin{proposition}[Stable coherence]
		\label{prop:grobner-stably-coherent}
		Every Gr\"obner ring is stably coherent.
	\end{proposition}
	
	This result is proved in \cite{PolaYengui2014} and is also recorded
	in \cite[Proposition~224]{Yengui2015Book}. It is particularly
	relevant when the Gr\"obner property is compared with the
	$1$-Gr\"obner property. The latter controls leading term ideals in
	one variable but does not by itself guarantee coherence.
	
	Valuation rings provide a natural setting in which this distinction
	can be made explicit. We allow valuation rings to have zero
	divisors.
	
	\begin{definition}[Valuation rings and the archimedean property]
		\label{def:valuation-ring-archimedean}
		A commutative ring $V$ is called a valuation ring if any two
		elements of $V$ are comparable under divisibility. A valuation
		ring without zero divisors is called a valuation domain.
		
		A valuation ring $V$ is called archimedean if, for every
		$a,b\in\operatorname{Rad}(V)\setminus\{0\}$, there exists
		$k\geq 1$ such that $a$ divides $b^k$.
	\end{definition}
	
	For a valuation domain, the archimedean property is equivalent to
	having Krull dimension at most one. If a valuation ring contains a
	nonzero zero divisor, then it is archimedean if and only if it has
	Krull dimension zero. See
	\cite[Propositions~264 and~265]{Yengui2015Book}.
	
	\begin{proposition}[Valuation domains]
		\label{prop:valuation-domain-grobner}
		Let $V$ be a valuation domain. The following conditions are
		equivalent.
		\begin{enumerate}
			\item The ring $V$ is $1$-Gr\"obner.
			\item The ring $V$ is a Gr\"obner ring.
			\item The ring $V$ is archimedean.
			\item The Krull dimension of $V$ is at most one.
		\end{enumerate}
	\end{proposition}
	
	The equivalence between the one variable condition and the dimension
	condition follows from
	\cite{LombardiSchusterYengui2012}. The multivariable lexicographic
	classification is proved in \cite{Yengui2014Lex}. Thus the
	$1$-Gr\"obner and Gr\"obner properties are equivalent for valuation
	domains.
	
	For valuation rings with zero divisors, coherence is an additional
	and essential condition.
	
	\begin{proposition}[Valuation rings]
		\label{prop:valuation-ring-grobner}
		Let $V$ be a valuation ring, possibly with zero divisors.
		\begin{enumerate}
			\item The ring $V$ is $1$-Gr\"obner if and only if it is
			archimedean.
			\item The ring $V$ is a Gr\"obner ring if and only if it is
			coherent and archimedean.
		\end{enumerate}
		Consequently, if $V$ is coherent, then $V$ is $1$-Gr\"obner if
		and only if it is a Gr\"obner ring.
	\end{proposition}
	
	The first assertion is proved in \cite{LiLiuZheng2017}, while the
	second is the classification obtained in
	\cite{PolaYengui2014}. In particular, if $V$ contains a nonzero zero
	divisor, then $V$ is $1$-Gr\"obner if and only if it has Krull
	dimension zero, whereas it is Gr\"obner if and only if it has Krull
	dimension zero and is coherent.
	
	The coherence assumption cannot be removed. Li, Liu, and Zheng
	construct a valuation ring that is $1$-Gr\"obner but not coherent
	\cite{LiLiuZheng2017}. By
	Proposition~\ref{prop:valuation-ring-grobner}, this ring is not
	Gr\"obner. Hence, without coherence, the implication
	\[
	1\text{-Gr\"obner}
	\quad\Longrightarrow\quad
	\text{Gr\"obner}
	\]
	is false even within the class of valuation rings.

The preceding results clarify the passage from the $1$-Gr\"obner
property to the Gr\"obner property. They show that this passage holds
for valuation domains and for coherent valuation rings, but can fail
without coherence when zero divisors are present. The dependence of
leading term finiteness on the chosen monomial order is a separate
question. These two problems are brought together in the Leading Terms
Ideals Conjecture. In the terminology of \cite{Yengui2015Book}, a ring
is strongly discrete if membership in finitely generated ideals is
decidable.

We now recall the original formulation of the conjecture.
	
	\begin{conjecture}[Leading Terms Ideals Conjecture]
		\label{conj:leading-terms-ideals}
		Let $R$ be a strongly discrete coherent ring. The following
		conditions were conjectured to be equivalent.
		\begin{enumerate}
			\item[(1)] For every $n\geq 1$, every monomial order $\prec$
			on $R[X_1,\ldots,X_n]$, and every finitely generated ideal
			$I\subseteq R[X_1,\ldots,X_n]$, the ideal $\LT_\prec(I)$ is
			finitely generated.
			\item[(2)] The ring $R$ is a Gr\"obner ring.
			\item[(3)] The ring $R$ is $1$-Gr\"obner.
		\end{enumerate}
	\end{conjecture}
	
	This is Conjecture~225 of \cite{Yengui2015Book}. Its unrestricted
	order independence assertion is false. Yengui constructed a
	valuation domain of Krull dimension one and a finitely generated
	polynomial ideal whose leading term ideal is not finitely generated
	with respect to a suitable irrational monomial order
	\cite{Yengui2021Counterexample}. We therefore replace arbitrary
	monomial orders in assertion~{\rm(1)} by rational monomial orders (see Conjecture~\ref{conj:updated-leading-terms}).

	The main theorem of this paper proves the equivalence of
	assertions~{\rm(1)} and~{\rm(2)} in
	Conjecture~\ref{conj:updated-leading-terms}. In fact, this
	equivalence holds over every commutative ring, without discreteness
	or coherence assumptions. Since every Gr\"obner ring is
	$1$-Gr\"obner by definition, the only remaining implication in the
	updated conjecture is
	\[
	R\text{ is }1\text{-Gr\"obner}
	\quad\Longrightarrow\quad
	R\text{ is Gr\"obner}.
	\]
	This implication holds for valuation domains and for coherent
	valuation rings by
	Propositions~\ref{prop:valuation-domain-grobner}
	and~\ref{prop:valuation-ring-grobner}. The example in
	\cite{LiLiuZheng2017} shows that coherence is indispensable if zero
	divisors are allowed.
	
	\subsection{Matrix representations of monomial orders}
	
	For $u,v\in\mathbb{R}^m$, we write
	$u<_{\mathrm{lex}}v$ if the first nonzero coordinate of $v-u$ is
	positive. Given a matrix $B\in\mathbb{R}^{m\times n}$, define a
	relation on the monomials of $R[X_1,\ldots,X_n]$ by
	\[
	X^\alpha\prec_B X^\beta
	\quad\Longleftrightarrow\quad
	B\alpha<_{\mathrm{lex}}B\beta.
	\]
	We say that $B$ represents a monomial order $\prec$ if
	$\prec=\prec_B$.
	
	\begin{theorem}[Matrix representation of monomial orders]
		\label{thm:matrix-representation}
		Every monomial order in finitely many variables over $R$ is
		represented by a finite real matrix.
	\end{theorem}
	
	The monomial order case is due to Robbiano
	\cite{Robbiano1985,Robbiano1986}. The corresponding representation
	theorem for monomial preorders is developed in
	\cite{KemperTrungAnh2018}.
	
	\begin{definition}[Rational monomial order]
		\label{def:rational-order}
		A monomial order is called \emph{rational} if it admits a matrix
		representation with rational entries.
	\end{definition}
	
	\begin{proposition}[Nonnegative integral matrix criterion]
		\label{prop:nonnegative-matrix-order}
		Let $M\in\mathbb{N}^{m\times n}$. The relation
		\[
		X^\alpha\prec_M X^\beta
		\quad\Longleftrightarrow\quad
		M\alpha<_{\mathrm{lex}}M\beta
		\]
		defines a monomial order on $R[X_1,\ldots,X_n]$ if and only if
		$\rank(M)=n$.
	\end{proposition}
	
	This is the full rank case of the matrix criterion described in
	\cite{KemperYengui2020,GuyotYengui2024}.
	
	\begin{theorem}[Integral normalization of rational monomial orders]
		\label{thm:integral-normalization}
		Every rational monomial order on $R[X_1,\ldots,X_n]$ is represented
		by a matrix $M\in\mathbb{N}^{n\times n}$ of rank $n$.
	\end{theorem}
	
	The normalization is recorded in
	\cite{KemperYengui2020,GuyotYengui2024}. A rational representing
	matrix may first be replaced by one with nonnegative integer entries.
	Since a monomial order has no ties between distinct monomials, the
	resulting matrix has rank $n$. Reading the rows from top to bottom,
	rows lying in the span of the preceding retained rows are
	lexicographically redundant and may be removed. This leaves exactly
	$n$ independent rows.
	
	\subsection{Group gradings and homogeneous ideals}
	
	We recall the basic terminology concerning group gradings and
	homogeneous ideals. The definitions of a group graded ring,
	homogeneous elements, homogeneous components, and graded submodules
	may be found in
	\cite[Chapter~A, Section~I.1, pp.~1--3]
	{NastasescuVanOystaeyen1982}. All statements in this subsection are
	valid over an arbitrary commutative ring with identity and require
	neither a domain nor a Noetherian hypothesis.
	
	\begin{definition}[Group grading]
		\label{def:group-grading}
		Let $G$ be an abelian group, written additively. A $G$-grading on a
		ring $S$ is a direct sum decomposition
		\[
		S=\bigoplus_{d\in G}S_d
		\]
		into additive subgroups such that $S_dS_e\subseteq S_{d+e}$ for all
		$d,e\in G$. A nonzero element of $S_d$ is called homogeneous of
		degree $d$.
		
		Every element $F\in S$ has a unique expression
		$F=\sum_{d\in G}F_d$, where $F_d\in S_d$ and only finitely many
		$F_d$ are nonzero. The nonzero elements $F_d$ are called the
		homogeneous components of $F$.
	\end{definition}
	
	The following elementary examples fix the conventions used below.
	
	\begin{example}[Standard gradings]
		\label{ex:standard-gradings}
		Let $R$ be a commutative ring.
		\begin{enumerate}
			\item Every ring $S$ has the trivial $G$-grading defined by
			$S_0=S$ and $S_d=0$ for $d\neq 0$.
			
			\item The polynomial ring $R[X_1,\ldots,X_n]$ has the standard
			$\mathbb Z$-grading defined by $\deg(X_i)=1$. Its component of
			degree $d\geq 0$ is the $R$-module generated by the monomials
			$X^\alpha$ with $|\alpha|=d$, and its components of negative
			degree are zero.
			
			\item The same polynomial ring has the fine
			$\mathbb Z^n$-grading defined by $\deg(X_i)=e_i$, where
			$e_1,\ldots,e_n$ is the standard basis of $\mathbb Z^n$. Thus
			$\deg(X^\alpha)=\alpha$.
		\end{enumerate}
	\end{example}
	
	The trivial grading is recorded in
	\cite[Chapter~A, Example~1.1.2(2), p.~2]
	{NastasescuVanOystaeyen1982}. The fine grading of a polynomial ring
	appears in \cite[p.~4]{MillerSturmfels2005}, and the more general
	degree-map description of multigradings is given in
	\cite[Definition~8.1, pp.~149--150]{MillerSturmfels2005}.
	
	The standard and fine gradings are special cases of gradings defined
	by an integer matrix. This is the form that will be used in the
	reduction argument.
	
	\begin{definition}[Matrix grading]
		\label{def:matrix-grading}
		Let $R$ be a commutative ring, let
		$T=R[Y_1,\ldots,Y_r]$, and let
		$W\in\mathbb Z^{k\times r}$. Denote the columns of $W$ by
		$w_1,\ldots,w_r$. Place $R$ in degree zero and assign
		$\deg_W(Y_i)=w_i$. Then, for $a\in\mathbb N^r$, one has
		$\deg_W(Y^a)=Wa$.
		
		For $d\in\mathbb Z^k$, the homogeneous component of degree $d$ is
		\[
		T_d=
		\bigoplus_{\substack{a\in\mathbb N^r\\Wa=d}}RY^a.
		\]
		These components define a $\mathbb Z^k$-grading
		$T=\bigoplus_{d\in\mathbb Z^k}T_d$. A polynomial is called
		$W$-homogeneous if all terms occurring in it have the same
		$W$-degree.
	\end{definition}
	
	For polynomial rings over a field, the construction obtained by
	assigning prescribed degrees to the indeterminates is given in
	\cite[Proposition~4.1.1, pp.~16--17]
	{KreuzerRobbiano2005}, and matrix gradings are introduced explicitly
	in \cite[Definition~4.1.6, p.~18]{KreuzerRobbiano2005}. The same
	construction is valid over an arbitrary commutative coefficient ring.
	The standard $\mathbb Z$-grading corresponds to
	$W=(1,\ldots,1)$, while the fine $\mathbb Z^r$-grading corresponds to
	$W=I_r$.
	
	We shall use ideals that are compatible with a given grading.
	
	\begin{definition}[Graded ideal]
		\label{def:graded-ideal}
		Let $S=\bigoplus_{d\in G}S_d$ be a $G$-graded ring. An ideal
		$H\subseteq S$ is called graded, or homogeneous, if
		$H=\bigoplus_{d\in G}(H\cap S_d)$.
	\end{definition}
	
	The following standard characterization will be used twice: first to
	show that an ideal generated by homogeneous elements is graded, and
	then to replace an element of a graded ideal by one of its homogeneous
	components.
	
	\begin{proposition}[Characterizations of graded ideals]
		\label{prop:graded-ideal-characterization}
		Let $H$ be an ideal of a $G$-graded ring $S$. The following
		conditions are equivalent.
		\begin{enumerate}
			\item The ideal $H$ is graded.
			\item Every homogeneous component of every element of $H$
			belongs to $H$.
			\item The ideal $H$ is generated, as an ideal of $S$, by
			homogeneous elements.
		\end{enumerate}
	\end{proposition}
	
	The equivalence of the first two conditions is the standard
	characterization of graded submodules
	\cite[Chapter~A, Section~I.1, p.~3]
	{NastasescuVanOystaeyen1982}. The complete equivalence, including the
	characterization by homogeneous ideal generators, is stated explicitly
	in \cite[Section~1.1.5, p.~20]{Hazrat2016}.
	
	Since every monomial is homogeneous for a matrix grading, the preceding
	proposition has the following immediate consequence.
	
	\begin{corollary}
		\label{cor:component-containing-leading-term}
		Let $S=R[Y_1,\ldots,Y_r]$ be equipped with a matrix grading, let
		$H\subseteq S$ be a graded ideal, and let $\prec$ be a monomial
		order on $S$. For $0\neq Q\in H$, let $P$ be the homogeneous
		component of $Q$ containing $\LT_\prec(Q)$. Then $P\in H$ and
		$\LT_\prec(P)=\LT_\prec(Q)$.
	\end{corollary}
	
	\section{Reduction to the lexicographic order}\label{sec:reduction}
	
	Throughout this section, let $R$ be a commutative ring with identity.
	Set $A=R[X_1,\ldots,X_n]$ and fix a rational monomial order
	$\prec$ on $A$. By the integral normalization theorem, choose a
	representing matrix $M$. We take $M$ to be nonnegative, integral,
	square, and of rank $n$. Thus
	\[
	X^\alpha\prec X^\beta
	\quad\Longleftrightarrow\quad
	M\alpha<_{\mathrm{lex}}M\beta.
	\]
	
	We introduce new variables $U_1,\ldots,U_n$ and set
	\[
	S=R[U_1,\ldots,U_n,X_1,\ldots,X_n].
	\]
	The ring $S$ is equipped with the lexicographic monomial order
	$\prec_{\mathrm{lex}}$ determined by
	\[
	U_1\succ\cdots\succ U_n\succ X_1\succ\cdots\succ X_n.
	\]
	Thus the entire $U$ block is compared before the $X$ block, the order
	of $U_1,\ldots,U_n$ agrees with the order of the rows of $M$, and the
	order within the $X$ block agrees with the lexicographic convention
	used in the definition of a Gr\"obner ring. The proof would remain
	valid for either ordering of the $X$ block, since comparisons of
	distinct monomials in a homogeneous component are decided in the
	$U$ block.
	
	We begin by defining the embedding that encodes the matrix comparison
	while retaining the original monomial. For
	$p=(p_1,\ldots,p_n)\in\mathbb N^n$, write
	$U^p=U_1^{p_1}\cdots U_n^{p_n}$. Define the $R$ algebra homomorphism
	$\Psi_M\colon A\to S$ by
	\begin{equation}
		\label{eq:tagged-embedding}
		\Psi_M(X^\alpha)= U^{M\alpha}X^\alpha.
	\end{equation}
	We also define $\chi\colon S\to A$ by
	$\chi(X_i)=X_i$ and $\chi(U_j)=1$. Then
	$\chi\circ\Psi_M=\operatorname{id}_A$, so $\Psi_M$ is
	injective. The factor $X^\alpha$ retains the original exponent vector,
	whereas $U^{M\alpha}$ records the comparison determined by $M$.
	
%
%
	
	Let $I=\langle f_1,\ldots,f_s\rangle\subseteq A$ be finitely
	generated. Let $H$ be the ideal of $S$ generated by
	$\Psi_M(f_1),\ldots,\Psi_M(f_s)$. To control this enlarged ideal,
	assign to a monomial $ U^pX^\alpha$ the degree
	$\deg_M( U^pX^\alpha)=p-M\alpha\in\mathbb Z^n$. This degree is
	additive under multiplication and therefore defines a $\mathbb Z^n$
	grading on $S$. In particular, every tagged monomial
	$ U^{M\alpha}X^\alpha$ has degree zero.
	
	\begin{proposition}[Homogeneous lifts]
		\label{prop:homogeneous-lifts}
		With the notation above, the following statements hold.
		\begin{enumerate}
			\item For every $f\in A$, the polynomial $\Psi_M(f)$ is homogeneous
			of degree zero.
			\item The ideal $H$ is graded.
			\item $\chi(H)=I$.
			\item Every $g\in I$ has the homogeneous lift $\Psi_M(g)\in H$, and
			$\chi(\Psi_M(g))=g$.
		\end{enumerate}
	\end{proposition}
	
	\begin{proof}
		Every term of $\Psi_M(f)$ has degree zero, which proves the first
		statement. The ideal $H$ is generated by homogeneous elements, so it
		is graded. Since $\chi(\Psi_M(f_i))=f_i$, one has
		$\chi(H)=I$. Finally, if $g=\sum_i q_i f_i$, then
		$\Psi_M(g)=\sum_i\Psi_M(q_i)\Psi_M(f_i)$ belongs to $H$, is
		homogeneous of degree zero, and dehomogenizes to $g$.
	\end{proof}
	
	The grading ties the $U$ exponents to the original exponent vectors.
	Consequently, the lexicographic order on each homogeneous component
	reproduces the original matrix order.
	
	\begin{lemma}[Monomials in a homogeneous component]
		\label{lem:homogeneous-component-structure}
		Let $P\in S$ be homogeneous. If $ U^pX^\alpha$ and
		$X^\beta U^r$ occur in $P$, then $p-r=M(\alpha-\beta)$. Equivalently,
		if $P$ has degree $d\in\mathbb Z^n$, every monomial in its support has
		the form $X^\alpha U^{d+M\alpha}$, where
		$d+M\alpha\in\mathbb N^n$.
	\end{lemma}
	
	\begin{proof}
		The two monomials have the same degree, so
		$p-M\alpha=r-M\beta$.
	\end{proof}
	
	\begin{lemma}[Exact dehomogenization of leading terms]
		\label{thm:exact-leading-term-correspondence}
		Let $0\neq P\in S$ be homogeneous for the grading $\deg_M$. Then
		$\chi(P)\neq0$ and
		\begin{equation}
			\label{eq:exact-leading-term-correspondence}
			\chi\!\left(\LT_{\prec_{\mathrm{lex}}}(P)\right)
			=
			\LT_\prec\!\left(\chi(P)\right).
		\end{equation}
	\end{lemma}
	
	\begin{proof}
		Suppose that two monomials $ U^pX^\alpha$ and $U^rX^\beta $ in the
		support of $P$ have the same image under $\chi$. Then
		$\alpha=\beta$, and Lemma~\ref{lem:homogeneous-component-structure}
		gives $p=r$. Thus distinct support monomials do not collide under
		$\chi$, and their coefficients remain unchanged. Hence
		$\chi(P)\neq0$.
		
		Now take two distinct monomials $ U^pX^\alpha$ and $U^rX^\beta $ in
		the support of $P$. Their $U$ exponent vectors cannot be equal,
		because $p=r$ would imply $M\alpha=M\beta$, and the rank condition on
		$M$ would give $\alpha=\beta$. Their comparison under
		$\prec_{\mathrm{lex}}$ is therefore decided in the $U$ block. By
		Lemma~\ref{lem:homogeneous-component-structure}, one has
		$p<_{\mathrm{lex}}r$ if and only if
		$M\alpha<_{\mathrm{lex}}M\beta$, which is equivalent to
		$X^\alpha\prec X^\beta$. The largest monomial in the support of $P$
		therefore dehomogenizes to the largest monomial in the support of
		$\chi(P)$, with the same coefficient. This proves
		\eqref{eq:exact-leading-term-correspondence}.
	\end{proof}
	
	\begin{remark}[The role of homogeneity]
		\label{rem:homogeneity-essential}
		The homogeneity assumption cannot be omitted. For $n=1$ and $M=(1)$,
		consider $P=U^2+XU$ with $U\succ X$. Then
		$\chi(\LT_{\prec_{\mathrm{lex}}}(P))=1$, whereas
		$\LT_\prec(\chi(P))=\LT_\prec(1+X)=X$.
	\end{remark}
	
	We can now descend finite generation from the enlarged ideal to the
	original ideal. The gradedness of $H$ allows each leading term
	generator to be represented by a homogeneous element before
	Lemma~\ref{thm:exact-leading-term-correspondence} is applied.
	
	\begin{theorem}[Reduction to the lexicographic order]
		\label{thm:fixed-ideal-reduction}
		Let $I=\langle f_1,\ldots,f_s\rangle\subseteq A$, and let $H\subseteq
		S$ be the ideal constructed above. If
		$\LT_{\prec_{\mathrm{lex}}}(H)$ is finitely generated, then
		$\LT_\prec(I)$ is finitely generated. Equivalently, if $H$ admits a
		finite Gr\"obner basis with respect to $\prec_{\mathrm{lex}}$, then
		$I$ admits a finite Gr\"obner basis with respect to $\prec$.
	\end{theorem}
	
	\begin{proof}
		The zero ideal is immediate, so assume $I\neq0$. Since $H$ is
		finitely generated, Proposition~\ref{prop:finite-grobner-basis}
		and the hypothesis yield a finite Gr\"obner basis
		$Q_1,\ldots,Q_t$ of $H$. For each $j$, let $P_j$ be
		the homogeneous component of $Q_j$ that contains
		$\LT_{\prec_{\mathrm{lex}}}(Q_j)$. Since $H$ is graded, one has
		$P_j\in H$ and
		$\LT_{\prec_{\mathrm{lex}}}(P_j)=
		\LT_{\prec_{\mathrm{lex}}}(Q_j)$. Thus the leading terms of
		$P_1,\ldots,P_t$ still generate
		$\LT_{\prec_{\mathrm{lex}}}(H)$.
		
		Set $p_j=\chi(P_j)$. Then $p_j\in I$, and
		Lemma~\ref{thm:exact-leading-term-correspondence} gives
		$p_j\neq0$ and
		$\LT_\prec(p_j)=
		\chi(\LT_{\prec_{\mathrm{lex}}}(P_j))$. Let $0\neq g\in I$. Proposition~\ref{prop:homogeneous-lifts} shows
		that $\Psi_M(g)$ is a nonzero homogeneous element of $H$. Since the
		leading terms of the $P_j$ generate
		$\LT_{\prec_{\mathrm{lex}}}(H)$, there exist
		$A_1,\ldots,A_t\in S$ such that
		\[
		\LT_{\prec_{\mathrm{lex}}}(\Psi_M(g))
		=
		\sum_{j=1}^t A_j\LT_{\prec_{\mathrm{lex}}}(P_j).
		\]
		Applying $\chi$ and using
		Lemma~\ref{thm:exact-leading-term-correspondence} gives
		$\LT_\prec(g)=\sum_{j=1}^t\chi(A_j)\LT_\prec(p_j)$.
		Therefore $\LT_\prec(I)\subseteq
		\langle\LT_\prec(p_1),\ldots,\LT_\prec(p_t)\rangle$. The reverse
		inclusion follows from $p_j\in I$, and hence
		\begin{equation}
			\label{eq:leading-term-ideal-descent}
			\LT_\prec(I)
			=
			\left\langle\LT_\prec(p_1),\ldots,\LT_\prec(p_t)\right\rangle.
		\end{equation}
		Finally, the finite set
		$\{f_1,\ldots,f_s,p_1,\ldots,p_t\}$ generates $I$, and its leading
		terms generate $\LT_\prec(I)$. It is therefore a finite Gr\"obner
		basis of $I$.
	\end{proof}
	
	The fixed ideal reduction immediately gives the corresponding global
	statement. The increase in the number of variables is the reason that
	the lexicographic hypothesis must be imposed in every finite number
	of variables.
	
	\begin{theorem}
		\label{thm:rational-characterization-grobner-rings}
		Let $R$ be a commutative ring with identity. Then the following
		conditions are equivalent.
		\begin{enumerate}
			\item[(1)] For every $n\geq 1$, every rational monomial order
			$\prec$ on $R[X_1,\ldots,X_n]$, and every finitely generated
			ideal $I\subseteq R[X_1,\ldots,X_n]$, the leading term ideal
			$\LT_\prec(I)$ is finitely generated.
			
			\item[(2)] The ring $R$ is a Gr\"obner ring.
		\end{enumerate}
	\end{theorem}
	
	\begin{proof}
		Assume that $R$ is a Gr\"obner ring. Let $n\geq 1$, let $\prec$
		be a rational monomial order on $R[X_1,\ldots,X_n]$, and let
		$I\subseteq R[X_1,\ldots,X_n]$ be a finitely generated ideal.
		Choose a matrix $M\in\mathbb N^{n\times n}$ of rank $n$ representing
		$\prec$, and form the ideal $H$ in
		$S=R[U_1,\ldots,U_n,X_1,\ldots,X_n]$. Set $Z_i=U_i$ for
		$1\leq i\leq n$ and $Z_{n+i}=X_i$ for $1\leq i\leq n$. Then
		$S=R[Z_1,\ldots,Z_{2n}]$, and the order
		$U_1\succ\cdots\succ U_n\succ X_1\succ\cdots\succ X_n$ is exactly
		the lexicographic order $Z_1\succ\cdots\succ Z_{2n}$. By
		Definition~\ref{def:grobner-ring-book},
		$\LT_{\prec_{\mathrm{lex}}}(H)$ is finitely generated.
		Theorem~\ref{thm:fixed-ideal-reduction} then gives finite generation
		of $\LT_\prec(I)$.
		
		Conversely, the lexicographic order
		$X_1\succ\cdots\succ X_n$ is rational and is represented by the
		identity matrix. Hence the second condition implies the
		lexicographic finite generation property in every finite number of
		variables, which is precisely the Gr\"obner ring property.
	\end{proof}

\begin{remark}[Uniformity in the number of variables]
	\label{rem:scope-reduction}
	Theorem~\ref{thm:rational-characterization-grobner-rings} is
	global with respect to the number of variables. Indeed, a rational
	monomial order on $R[X_1,\ldots,X_n]$ is reduced to a
	lexicographic problem in
	$R[U_1,\ldots,U_n,X_1,\ldots,X_n]$, which has $2n$ variables.
	Thus the lexicographic finiteness property is required in every
	finite number of variables.
	
	This differs from Conjecture~3 of Pola and Yengui
	\cite{PolaYengui2014}, which concerns a fixed number of variables
	and asks whether the property for one elimination order implies the
	same property for every monomial order on the same polynomial ring.
	The present construction does not show that the lexicographic
	property on $R[X_1,\ldots,X_n]$ alone implies the corresponding
	property for rational monomial orders on that same ring.
\end{remark}
	
	\section{Consequences for valuation rings}
	\label{sec:valuation-rings}
	
	We now apply the main theorem to the known classifications of
	Gr\"obner valuation rings. Following Pola and Yengui, a valuation ring
	is a commutative ring in which any
	two elements are comparable under divisibility, and zero divisors are
	allowed. A valuation ring without zero divisors is called a valuation
	domain \cite{PolaYengui2014}.
	
	For valuation domains, Yengui's lexicographic theorem identifies the
	Gr\"obner ring property with Krull dimension at most one
	\cite{Yengui2014Lex}. The cited paper writes the variables in the
	reverse displayed priority. Reversing their names gives the
	lexicographic convention $X_1\succ\cdots\succ X_n$ used here. We
	therefore obtain the following consequence for rational monomial
	orders.
	
	\begin{corollary}[Valuation domains]
		\label{cor:rational-orders-valuation-domains}
		Let $V$ be a valuation domain. The following conditions are
		equivalent.
		\begin{enumerate}
			\item $V$ is a Gr\"obner ring.
			\item $\dim V\leq 1$.
			\item For every $n\geq 1$, every rational monomial order $\prec$ on
			$V[X_1,\ldots,X_n]$, and every finitely generated ideal
			$I\subseteq V[X_1,\ldots,X_n]$, the ideal $\LT_\prec(I)$ is finitely
			generated.
		\end{enumerate}
	\end{corollary}
	
	\begin{proof}
		The equivalence of the first two conditions is Yengui's lexicographic
		classification. The equivalence of the first and third conditions is
		Theorem~\ref{thm:rational-characterization-grobner-rings}.
	\end{proof}
	
	The same argument applies to valuation rings with zero divisors. A
	ring is coherent if every finitely generated ideal is finitely
	presented. A valuation ring $V$ is archimedean if, for every nonzero
	$a,b\in\operatorname{Rad}(V)$, there exists $k\geq 1$ such that
	$a$ divides $b^k$.
	Pola and Yengui proved that a valuation ring is a Gr\"obner ring if
	and only if it is coherent and archimedean. Their proof of the
	sufficient direction uses the lexicographic order, so the
	classification applies to the convention adopted here
	\cite{PolaYengui2014}.
	
	\begin{corollary}[Valuation rings]
		\label{cor:rational-orders-valuation-rings}
		Let $V$ be a valuation ring, possibly with zero divisors. The
		following conditions are equivalent.
		\begin{enumerate}
			\item $V$ is a Gr\"obner ring.
			\item The ring $V$ is coherent and archimedean.
			\item For every $n\geq 1$, every rational monomial order $\prec$ on
			$V[X_1,\ldots,X_n]$, and every finitely generated ideal
			$I\subseteq V[X_1,\ldots,X_n]$, the ideal $\LT_\prec(I)$ is finitely
			generated.
		\end{enumerate}
	\end{corollary}
	
	\begin{proof}
		The equivalence of the first two conditions is the classification of
		Pola and Yengui. The equivalence of the first and third conditions is
		Theorem~\ref{thm:rational-characterization-grobner-rings}.
	\end{proof}
	
	Equivalently, the second condition in
	Corollary~\ref{cor:rational-orders-valuation-rings} may be replaced by
	the requirement that either $V$ is a valuation domain of Krull
	dimension at most one, or $V$ has Krull dimension zero and
	$\operatorname{Ann}_V(a)$ is finitely generated for every $a\in V$
	\cite{PolaYengui2014}.
	
	Lombardi, Neuwirth, and Yengui cite an unpublished 2024 preprint by
	Yengui entitled \emph{A solution to the Gr\"obner ring conjecture}
	and describe it as a recent solution
	\cite{LombardiNeuwirthYengui2026}. Since that manuscript was not
	available to us at the time of writing, a comparison of the two
	approaches is beyond the scope of the present paper. We therefore
	make no claim of priority for the valuation-ring consequences above.
	Our main contribution is the general rational characterization of
	Gr\"obner rings over arbitrary commutative coefficient rings.
	
	We finish by explaining the restriction to rational monomial orders.
	The construction in Section~\ref{sec:reduction} starts with a finite
	nonnegative integral matrix representation of the order. Such a
	matrix supplies the exponents in the tagged embedding
	$\Psi_M(X^\alpha)= U^{M\alpha}X^\alpha$. An irrational monomial order
	has no finite rational matrix representation, so this construction is
	not available.
	
	The restriction is not merely a limitation of the proof. Yengui
	constructed a valuation domain of Krull dimension one and a finitely
	generated ideal in two variables whose leading term ideal is not
	finitely generated for a suitable irrational monomial order
	\cite{Yengui2021Counterexample}. Such a valuation domain is a
	Gr\"obner ring by the lexicographic classification, but it does not
	satisfy the corresponding property for all monomial orders. Thus the
	rationality hypothesis in
	Theorem~\ref{thm:rational-characterization-grobner-rings} cannot be
	removed in general.
	
	\section{Conclusion}
	\label{sec:conclusion}
	
	Using the lexicographic definition of Gr\"obner rings, we have proved
	that a commutative ring $R$ is a Gr\"obner ring if and only if, in
	every finite number of variables, the leading term ideal of every
	finitely generated polynomial ideal is finitely generated with
	respect to every rational monomial order. This gives the rational
	monomial order form of the equivalence between the first two
	conditions in the Leading Terms Ideals Conjecture. The remaining
	comparison with the one-variable condition is a different problem and
	is not settled by the present reduction.
	
	The proof is obtained from a fixed-ideal construction. A rational
	order represented by a nonnegative integral matrix $M$ is encoded by
	the tagged embedding $\Psi_M(X^\alpha)= U^{M\alpha}X^\alpha$. The
	associated grading $\deg_M( U^pX^\alpha)=p-M\alpha$ makes
	dehomogenization collision free on homogeneous components and gives
	an exact leading term identity. Applying dehomogenization to an ideal
	membership relation then descends finite generation from the enlarged
	lexicographic ideal to the original ideal. The argument uses neither
	Noetherianity nor the absence of zero divisors.
	
	The theorem also clarifies the role of rationality in the theory of
	Gr\"obner rings. The lexicographic property automatically extends to
	all rational monomial orders, but it need not extend to irrational
	orders. Together with the known classifications of valuation domains
	and valuation rings, the result yields the corresponding rational
	order classifications. The irrational counterexample shows that an unrestricted statement
	covering all monomial orders is impossible in general. Thus the
	restriction to rational monomial orders is mathematically
	substantive.
	
	\section*{Acknowledgments}
	This research was supported by the STU Scientific Research Initiation
	Grant under No. NTF24023T.
	\bibliographystyle{elsarticle-num-names-alpha}
	\bibliography{references}
	
\end{document}